\documentclass[a4paper,12pt]{article}

\usepackage{amsmath,amsthm,amssymb,indentfirst,graphicx,caption,subcaption,setspace}
\usepackage[english]{babel}
\usepackage[big]{layaureo}
\usepackage{mathtools}
\usepackage{booktabs}
\theoremstyle{plain}
\usepackage{bbm}
\usepackage{appendix}
\newtheorem{proposition}{Proposition}
\newtheorem{remark}{Remark}
\newtheorem{lemma}{Lemma}

\newtheorem{definition}{Definition}
\newtheorem{assumption}[definition]{Assumption}
\usepackage[usenames]{color}
\usepackage{tikz}
\definecolor{red}{rgb}{1.0,0.0,0.0}

\definecolor{blu}{rgb}{0.0,0.0,1.0}

\definecolor{gre}{rgb}{0.03,0.50,0.03}

\definecolor{darkviolet}{rgb}{0.58, 0.0, 0.83}

\title{Optimal Planning in Habit Formation Models with Multiple Goods}
\author{Mauro Bambi\thanks{Economics and Finance Department, Durham University, Durham, UK. email: mauro.bambi@durham.ac.uk}, \and Daria Ghilli \thanks{Università degli Studi di Pavia, Dipartimento di Science Economiche e Aziendali, Pavia, IT. email: daria.ghilli@unipv.it} ,\and Fausto Gozzi \thanks{Economics and Finance Department, LUISS Guido Carli University, Rome, IT. email: fgozzi@luiss.it}, \and Marta Leocata \thanks{Scuola Normale Superiore, Pisa, IT. email: marta.leocata@sns.it} }
\date{June 2022}

\begin{document}
	\maketitle
	\begin{abstract}
	In this paper, on the line e.g. of \cite{carroll2000saving}) we investigate
a model with habit formation and two types of substitute goods.
Such family of models, even in the case of 1 good, are difficult to study since their utility function is not concave in the interesting cases (see e.g. \cite{bambi2020internal}), hence the first order conditions are not sufficient.

We introduce and explain the model and provide some first results using the dynamic programming approach. Such results will form a solid ground over which a deep study of the features of the solutions can be performed.

		\medskip
		
		
	\end{abstract}

	\newpage
	
\tableofcontents
%\dg{Usare unica notazione $h_{h_0}^\ell$ dove serve invece che solo $h$ e specificarlo}\ml{io ho messo ovunque $h^{h_0,\ell}$}
%
%\dg{mettere $\ell(\dot)$  sotto il funzionale} \ml{fatto}
%
%\dg{mettere $s$ o $t$ in $\ell$ e $h$ nei funzionali} \ml{fatto}
%
%\dg{mettere nell'intro che $\gamma(\sigma-1)\geq 1$} 

\section{Introduction}
The main aim of this paper is to formulate and study a growth model with habits formation (on the line e.g. of \cite{carroll2000saving, carroll1997comparison})
which, differently from the above papers, takes into account the presence of two types of goods and where the perspective of the social planner is embraced. 

It has been observed the crucial relation between internal habit formation and increase in savings in a particular parameters' set, where the multiplicative utility function is never jointly concave in consumption and habits. In particular,  joint concavity is never possible when the coefficient of relative risk aversion is bigger than one, i.e. $\sigma >1$, and if agent's weighting habit is less than one, i.e. $\gamma<1$. The lack of concavity makes this class of model difficult to study, even in the case of one good (see e.g. \cite{bambi2020internal, yang2014rational, alonso2005growth, dechert2012complete, kamihigashi2007nonsmooth}).

This implies that sufficient optimality conditions of maximum principle type are missing, see page $103$ in \cite{seierstad1986optimal}.

Therefore, in this paper, we address the problem using the dynamic programming approach. We emphasize that the advantage of this approach is that it provides optimality conditions independent of the concavity assumption. Although this method is a rather powerful tool for providing uniqueness of the optimal control strategy, a uniqueness result is not demonstrated in this paper. We plan to present some results in this direction in some future work. 

This approach was developed in the seminal work of Crandall and Lions, see \cite{lions1981solutions, crandall1983viscosity} and it has found a recent application to macroeconomic problems (see  \cite{achdou2022income,freni2008multisector, calvia2023optimal}).
In this paper, we apply this theory using the homogeneity properties of the problem(for a similar approach see \cite{freni2008multisector, bambi2020internal,cannarsa1989generalized, BCD}). 

%%%%%%%%%%%%
The model under analysis is a variant of the model proposed in \cite{carroll2000saving} and later also studied in \cite{bambi2020internal}. The main difference from the works cited above is that we embrace the social planner's perspective in which the accumulation of goods is not taken into account. So, we consider the problem of a social planner who takes into account two goods to be consumed and only one of them is related to a habit stock. The
instantaneous utility function that we examine for habit-related consumption was introduced in \cite{abel}, and it is
\[\frac{\left(\frac{c_1}{h^\gamma}\right)^{1-\sigma}}{1-\sigma},\]
while the utility function for the second good is the classical CRRA utility function
\[\frac{(c_2)^{1-\sigma}}{1-\sigma}.\]
We stress the fact that the utility function composed by the sum of the two utility functions above, is not strictly concave when $\sigma>1$ and $\gamma<1$.
In some future projects, we aim to investigate also the impact of lockdown on the optimization problem. 

Before analyzing the complete model, we propose some results on a further simplified model where the utility of consumption in the second goods is neglected. Then, by Dynamic Programming, we furnish some results on the complete model: we prove that the value function is a viscosity solution of the HJB equation, we derive some properties on the qualitative behavior of the value function.
We think that such results constitute a reliable basis to study in a satisfactory way the properties of the optimal paths.
%%%%%%%%%%%%%%
The paper is organized as follows. In Section 2, the optimal control problem is introduced. In Section 3, we introduce a simplification of the optimal control problem and we prove that if the set of controls is unbounded, i.e. $(\underline \ell,+\infty)$, then the value function is trivially null and moreover we furnish some results suggesting that if the set of controls is bounded, i.e. $(\underline \ell,\bar \ell)$, then corner solutions are obtained. In Section 4, first the Maximum Principle is stated and existence and uniqueness of the steady state are proven. Then by dynamic programming approach, we prove that the value function $V$ is a viscosity solution of the HJB equation and we also derive some qualitative properties of $V$, namely we prove that $V$ is negative, decreasing, locally lipschitz and $V(+\infty)=-\infty$.

\section{Model Setup}

A social planner faces the following problem:	
	$$\max_{c_1,c_2}  \int_0^\infty e^{-\rho t} U(c_1(t),c_2(t),h^{c_1,h_0}(t)) dt$$
	subject to the following constraints
	\begin{eqnarray}
		&& c_1=A\ell_1	\label{c1}\\
		&& c_2=B\ell_2	\label{c1}\\
		&& \dot {h}=\phi( c_1 - h) \ \ with \ h(0)=h_{0} \label{habits_state0} \\
		&& \ell_1+\ell_2=1 \\
		&& \ell_1\in(\underline{\ell},\bar \ell)
	\end{eqnarray}
Renaming $\ell_1\equiv\ell$ and setting $A=1$, the problem can be rewritten as

\begin{equation}\label{optproblem}
\max_{\ell(\cdot) \in \mathcal{A}}\int_0^\infty e^{-\rho t} U(\ell(t),B(1-\ell(t)),h^{\ell,h_0}(t)) dt
\end{equation}
where $\mathcal{A}$ is the set of admissible control (which will be specified in the following sections),
subject to the following constraints
\begin{eqnarray}
	&& \dot {h}=\phi( \ell - h) \ \ with \ h(0)=h_{0} \label{habits_state1} \\
	&& \ell\in(\underline{\ell},\bar \ell).
\end{eqnarray}
%%plus a lockdown specification which can be described by a random variable (see below).
%%
%%\medskip
%%
%%Cases to study:
%%\begin{itemize}
%%	\item Unanticipated lockdown of unananticipated length. During lockdown $\ell=\bar \ell$ meaning that there is a minimum provision of $c_2=B(1-\bar \ell)$, very simple as the economy under lockdown does not require an optimization problem;
%%	\item Unanticipated lockdown of unanticipated length $\tau$.
%%	\item Unanticipated lockdown of random length $\tau$.
%%	\item Random lockdown arrival and length: there is a certain probability that $\ell=\bar\ell$ for a random duration.
%%\end{itemize}
%
%The first three points are very similar to what we did in the other paper. The last point which is actually the most interesting is new. Perhaps we could model it having in mind a Markov chain without and with an absorbing state (case where eventually the probability of a lockdown is zero due to vaccination etc.)
%
%\medskip
%
%\textbf{Objective:} under which condition is it optimal (after the lockdown/pandemic) to have $c_2=\bar\ell$ or $c_2\rightarrow\bar\ell$ ?
%
%\medskip
%

We specify the following utility function :
$$U(c_1,c_2,h)= u(c_1,h)+v(c_2)=\frac{\left(\frac{c_1}{h^\gamma}\right)^{1-\sigma}}{1-\sigma}+\frac{c_2^{1-\sigma}}{1-\sigma}$$
The parameter $\gamma$ signifies the importance of habits in the utility function. When $\gamma = 0$, habits have no influence, and utility is solely determined by the level of consumption. Conversely, when $\gamma = 1$, habits matter as much as consumption in determining utility. For intermediate values of $\gamma$, where $\gamma \in (0,1)$, habits have an effect on utility, but less so than consumption. Our analysis focuses on this last and more realistic case.

This functional form of the utility function was introduced by Carroll et al \cite{carroll1997comparison} to solve several issues arising with another popular utility function where habits enters in a subtractive form.

Notice that $u(c_1,h)$ is monotonically increasing in $c_1$ and concave in $c_1$ for any choice of the parameters, i.e. $u_{c_1}>0$ and $u_{c_1c_1}<0$ always. On the other hand, $u(c_1,h)$ is monotonically decreasing in the habits, i.e. $u_{h}<0$, meaning that habits are harmful.\footnote{More generally we could assume that $\gamma\in(-1,1)$ so that in the interval $\gamma\in(-1,0)$ the habits becomes beneficial since $u_h>0$.} Concavity of $u(c_1,h)$ with respect to $h$, i.e. $u_{hh}<0$, implies that 
\begin{equation}
\sigma>1+\frac1\gamma.
\end{equation}
 This condition prevents $u(c_1,h)$ to be jointly concave in $(c_1,h)$ as $det(D^2u(c_1,h))\geq 0$ if and only if $\sigma\leq\frac\gamma{1-\gamma}$ (see also \cite{bambi2020internal} page 10). Nevertheless this does not necessarily imply that the solution of the problem is not a maximum. Finally $v(c_2)$ is monotonically increasing and concave in $c_2$. To sum up, in all the work we will assume the following hypothesis holding on the coefficients.
\begin{assumption}
We assume that $\sigma>1$, $\gamma<1$ and $\gamma(\sigma-1)>1$.
\end{assumption}
Substituting the production function into the utility function we get:
$$U(\ell,B(1-\ell),h)\equiv u(\ell,h)+v(\ell)=\frac{\left(\frac{\ell}{h^\gamma}\right)^{1-\sigma}}{1-\sigma}+\tilde B\frac{(1-\ell)^{1-\sigma}}{1-\sigma}$$
where $\tilde B\equiv B^{1-\sigma}$. In the core of the text we will also use the following notation,
\[u(\ell,h)=u_1(h)u_2(\ell)\]
where $u_1(h)=\frac{h^{\gamma(\sigma-1)}}{1-\sigma}$, $u_2(\ell)=\ell^{1-\sigma}$.

In order to stress the relevance of such utility function, first we consider an simpler version per the second term is not considered. We consider

\[\widetilde{U}(\ell,h)= u(\ell,h)=\frac{\left(\frac{\ell}{h^\gamma}\right)^{1-\sigma}}{1-\sigma}\]

\section{Some results on a simpler model}
We consider the following optimal control problem for a social planner,

\begin{equation}\label{valuefunction_simple}
W(h_0)=\sup_{\ell(\cdot)\in \mathcal{A}}  \  \int_0^\infty e^{-\rho t} \widetilde{U}\left(\ell(t), h^{h_0,\ell}(t)\right) dt
\end{equation}
	subject to the following state equation
\begin{equation}\label{eq:h_simple}
 \dot {h}=\phi( \ell - h) \ \ with \ h(0)=h_{0}
 \end{equation}
and where the set of admissible control is,
\[\mathcal{A}=\{\ell\in L^1_{loc}(\mathbb{R}_+)\,\, \text{s.t.}\,\, \ell(t)\in[\underline{\ell},\bar \ell],\,\,\text{and}\,\, h^{h_0,\ell}(t)\geq 0,\,\,\forall t\geq 0\}.\]
Notice that the state constraint prescribed in the definition of the set $\mathcal{A}$ is somehow fictitious. Indeed, the state $h$ remains positive for each $\ell \in L^1_{loc}(\mathbb{R}_+)$. The set of admissible control can be rewritten simply as,
\begin{equation}\label{eq: admissible_controls}
\mathcal{A}=\{\ell\in L^1_{loc}(\mathbb{R}_+)\,\, \text{s.t.}\,\, \ell(t)\in[\underline{\ell},\bar \ell]\}.
\end{equation}
We first, study the problem by considering $\bar\ell=+\infty$.

\begin{lemma}\label{lem: prop_v_simple}
Let $W$ be the value function defined in \eqref{valuefunction_simple} with. Then
\begin{itemize}
\item[(i)]  $W$ is negative,
\item[(ii)] $W$ is decreasing.
\end{itemize}
\end{lemma}
\begin{proof}
(i) Since the utility function is negative, $W\leq 0$.\\
(ii) By the linearity of the state equation we get that if $h_1<h_2$, then $h^{h_1,\ell}(t)<h^{h_2,\ell}(t)$ for all $t\in \mathbb{R}_+$. Since the utility function, is decreasing with respect to the variable $h$, we conclude that
\[W(h_1)-W(h_2)>0.\]
\end{proof}

\begin{proposition}\label{prop:W_null}
If $\phi(1-\sigma)<\rho$ and $\bar\ell=+\infty$, then the value function defined in \eqref{valuefunction_simple} is null, i.e. $W\equiv 0$.
\end{proposition}
\begin{proof}
In order to prove the result, we just find a sequence of admissible controls, $c_n$ such that $J(c_n)\to 0$ and $c_n\to \infty$ for $n\to \infty$ . By choosing $c_n=n$, one can check that
\[J(c_n)= \int_0^\infty e^{(-\rho+\phi(1-\sigma)) t} \frac{n^{1-\sigma}}{1-\sigma}\left(h_0+n\frac{e^{\phi t}-1}{\phi} \right)^{\gamma(1-\sigma)} dt\to 0, \quad n\to \infty\]
\end{proof}

\begin{remark}
Notice that if we assume also that $\underline{\ell}=0$, then the value function $W$ is $(1-\gamma)(1-\sigma)$-homogenous. Indeed, by linearity of the state equation, we get that for each $\alpha>0$ and $h_0\in \mathbb{R}_+$,
\begin{eqnarray*}
W(\alpha h_0)&=& \sup_{\ell(\cdot)\in \mathcal{A}}   \int_0^\infty e^{-\rho t} \frac{\left(\ell(t)\right)^{1-\sigma}}{1-\sigma} \left(h^{\alpha h_0, \ell}(t)\right)^{-\gamma(1-\sigma)} dt\\
&=& \sup_{\ell(\cdot)\in \mathcal{A}}   \int_0^\infty e^{-\rho t} \frac{\left(\alpha \frac{\ell(t)}{\alpha}\right)^{1-\sigma}}{1-\sigma} \left(h^{\alpha h_0,\alpha\frac{\ell}{\alpha}}(t)\right)^{-\gamma(1-\sigma)} dt\\
&=& \sup_{\frac{\ell(\cdot)}{\alpha}\in \mathcal{A}}   \int_0^\infty e^{-\rho t} \frac{\left(\alpha \frac{\ell(t)}{\alpha}\right)^{1-\sigma}}{1-\sigma} \left(h^{\alpha h_0,\alpha\frac{\ell}{\alpha}}(t)\right)^{-\gamma(1-\sigma)} dt\\
&=& \sup_{\tilde{\ell}(\cdot)\in \mathcal{A}}  \int_0^\infty e^{-\rho t} \frac{\left(\alpha \tilde{\ell}(t)\right)^{1-\sigma}}{1-\sigma} \left(h^{\alpha h_0,\alpha\tilde{\ell}}(t)\right)^{-\gamma(1-\sigma)} dt\\
&=& \alpha^{(1-\gamma)(1-\sigma)}\sup_{\tilde\ell(\cdot)\in \mathcal{A}}   \int_0^\infty e^{-\rho t} \frac{\left( \tilde{\ell}(t)\right)^{1-\sigma}}{1-\sigma} \left(h^{ h_0,\tilde{\ell}}\right)^{-\gamma(1-\sigma)} dt=\alpha^{(1-\gamma)(1-\sigma)}W(h_0).\\
\end{eqnarray*}

Coherently with the result of Proposition \ref{prop:W_null}, we observe that the only function negative, decreasing and $(1-\gamma)(1-\sigma)$-homogenous is the null function.
\end{remark}

Note that in the paper by Carroll et al. \cite{carroll2000saving}, a zero-value function cannot be a solution because the agents in their model have the opportunity to postpone current consumption by investing in the capital stock, thereby increasing future utility. This particular channel is absent in our paper, setting our approach apart from theirs.

If $\bar \ell\neq +\infty$, the problem is more difficult to study. Although, we do not present a complete study for this problem, we will show some partial results that seems to confirm that in this framework the optimum is reached at $\ell(t)\equiv\overline{\ell}$. In the following proposition, we prove that if only constant controls are considered, our conjecture is true.

\begin{proposition}
Consider $\bar\ell\neq+\infty$, and as set of admissible control, $ \tilde{\mathcal{A}}\subset \mathcal{A}$
\[\tilde{\mathcal{A}}=\{\ell(t)\equiv l, \,\,l\in[\underline{\ell},\bar{\ell}]\}.\]
Then the problem
\begin{equation}\label{valuefunction_simple2}
\tilde{W}(h_0)=\sup_{\ell(\cdot)\in \tilde{\mathcal{A}}}  \  \int_0^\infty e^{-\rho t} \frac{\left(\frac{\ell(t)}{h(t)^\gamma}\right)^{1-\sigma}}{1-\sigma} dt
\end{equation}
subject to the state equation \eqref{eq:h_simple}, attains its maximum for $\ell=\bar{\ell}$, i.e. $\tilde{W}(h_0)=J(\bar{\ell})$.
\end{proposition}

\begin{proof}
The result is a consequence of the monotonicity of the function,
\[F(l)=\frac{l^{1-\sigma}}{1-\sigma}\left[e^{-\phi t}h_0+\phi l\int_0^t e^{-\phi(t-s)}ds\right].\]
We notice that
\[F'(l)=l^{-\sigma}(a(t)+b(t)l)^{\gamma(\sigma-1)}\left[\frac{a(t)+b(t)l(1-\gamma)}{a(t)+b(t)}\right],\]
with $a(t)=e^{-\phi t}h_0$, $b(t)=1-e^{-\phi t}$. Since $a(t), b(t)>0$ and $\gamma<1$, we have that $F$ is an increasing function, and so it is the objective function.
\end{proof}

\begin{proposition}
Let $W$ be the value function defined in \eqref{valuefunction_simple} with $\bar{\ell}\neq +\infty$. Then for $h_0\to \infty$,
\[c_2 \frac{(h_0\wedge \bar l)^{\gamma(\sigma-1)}}{1-\sigma}\leq W(h_0)\leq c_1\frac{h_0^{\gamma(\sigma-1)}}{1-\sigma} \]
\end{proposition}
\begin{proof}
First, we observe that
\begin{eqnarray*}
V(h_0)&\geq& \int_0^\infty e^{-\rho t} \frac{\left(\bar\ell\right)^{1-\sigma}}{1-\sigma}(h^{h_0,\bar \ell}(t))^{-\gamma(1-\sigma)} dt\\
&=& \int_0^\infty e^{-\rho t} \frac{\left(\bar\ell\right)^{1-\sigma}}{1-\sigma}(e^{-\phi t}(h_0-\bar \ell)+\bar \ell)^{-\gamma(1-\sigma)} dt
\end{eqnarray*}
Since if $h_0\geq \bar \ell$, then $\bar \ell\leq e^{-\phi t}(h_0-\bar \ell)+\bar \ell\leq h_0$ and if $h_0\leq \bar \ell$, then $h_0\leq e^{-\phi t}(h_0-\bar \ell)+\bar \ell\leq \bar \ell$, we get that
\[V(h_0)\geq \frac{\bar \ell^{1-\sigma}}{1-\sigma}\frac{\left(h_0\wedge \bar \ell\right)^{\gamma(\sigma-1)}}{\rho}.\]
Since $h^{h_0,\ell}\geq h_0$ and moreover the function $\frac{\ell^{1-\sigma}}{1-\sigma}$ is increasing and negative
\begin{eqnarray*}
V(h_0)&=&\sup_{\ell(\cdot) \in \tilde{\mathcal{A}}} \int_0^\infty e^{-\rho t} \frac{\ell(t)^{1-\sigma}}{1-\sigma}(h^{h_0,\ell}(t))^{-\gamma(1-\sigma)} dt\\
&\leq& \int_0^\infty e^{-\rho t} \frac{\left(\bar\ell\right)^{1-\sigma}}{1-\sigma}(h_0)^{-\gamma(1-\sigma)} dt\\
&=&  \frac{1}{\rho}\frac{\left(\bar\ell\right)^{1-\sigma}}{1-\sigma}(h_0)^{-\gamma(1-\sigma)}\\
\end{eqnarray*}

\end{proof}

\section{Some results on the optimization problem \eqref{optproblem}}

We now present some results on the richer model with two consumption goods. We start with some preliminary analysis based on the Maximum Principle. It is essential to note that while the steady-state results we will present adhere to the necessary conditions for optimality, we cannot ascertain whether they also meet the sufficiency conditions. This issue arises due to the non-concavity of the utility function.
Our future purpose is to investigate whether the dynamic programming approach helps to understand whether the optimal strategies, trajectories and costate, coincide with the unique solution of the Maximum Principle.

%We study here the planner problem and more precisely we will focus on a case of an interior solutoin before the lockdown. We will also focus on the case that before the lockdown the economy was at its steady state.
\subsection{Existence and uniqueness of steady state}
In order to study existence and uniqueness of the steady state, we derive the current value Hamiltonian:
\begin{eqnarray}\label{eq:H_CV}
H_{CV}(\ell,h,\mu)&\equiv&u(\ell,h)+v(\ell)+\mu\phi(\ell-h)\\ &=& \frac{\left(\frac{\ell}{h^\gamma}\right)^{1-\sigma}}{1-\sigma}+\tilde B\frac{(1-\ell)^{1-\sigma}}{1-\sigma}+\mu\phi(\ell-h)
\end{eqnarray}
which is not concave as the utility function $u(.)$ is not jointly concave in $(\ell,h)$. We apply the maximum principle, which we recall in the  following.
\begin{proposition}
Let $\ell$ be a solution of the optimal control problem \eqref{optproblem}. Then $\ell$ solves
\begin{eqnarray}\label{eq:maxprinc}
	&& \dot{\mu}=-u_h(\ell,h)+\mu(\phi+\rho)\\
&& u_\ell(\ell,h)+v_\ell(\ell)+\mu\phi=0\\
&&\dot{h}=\phi(\ell-h) \\
& &\lim_{t \to +\infty} h\mu e^{-\rho t}=0
\end{eqnarray}
\end{proposition}

Consider now a stationary steady state of the economy where $\dot h=\dot\mu=0$. From the state equation we have immediately that $h=\ell=h^*$ with $h^*$ indicating the steady state value.

\begin{proposition}
	A unique steady state $h=h^*$ always exists.
\end{proposition}
\begin{proof}
At the steady state, the two first order conditions of \eqref{eq:maxprinc} become:
\begin{eqnarray}
	&& u_\ell(h^*)+v_{\ell}(h^*)+\mu^*\phi=0 \\
	&& \mu^*=\frac{u_h(h^*)}{\phi+\rho}<0
\end{eqnarray}
Substituting the latter into the first condition leads to
$$u_\ell(h^*)+v_{\ell}(h^*)+\frac{\phi}{\phi+\rho}u_h(h^*)=0$$
which rewrites as it follows when we consider the CES utility function:
$$(h^*)^{-\sigma(1-\gamma)-\gamma}-\tilde B(1-h^*)^{-\sigma}-\frac{\phi\gamma}{\rho+\phi}(h^*)^{-\sigma(1-\gamma)-\gamma}=0$$
or equivalently (rearranging the terms):
$$\left(\frac{\tilde B(\rho+\phi)}{\rho+\phi(1-\gamma)}\right)^{\frac1\sigma}(h^*)^{\frac{\sigma(1-\gamma)+\gamma}\sigma}=1-h^*$$
Notice that the LHS of the equation is a function of $h^*$ which is zero at $h^*=0$ and is monotonically increasing as long as $\frac{\sigma(1-\gamma)+\gamma}\sigma>0$ or equivalently $\sigma>-\frac\gamma{1-\gamma}$ which is always respected since $\gamma\in(0,1)$ and $\sigma$ is positive. On the other hand the RHS of the equation is a linear decreasing function having value 1 when $h^*$ is zero. Therefore it always exists a unique $h^*$ when the two functions intersect each other.
\end{proof}

%To be investigated also whether habits are addictive or satiating, i.e. the sign of $\frac{dc_1}{d h}$. Notice that although there is separability in $(c_1,c_2)$ in the utility function, habits affect $c_2$ as well through $\ell$. Suppose for example that habits on $c_1$ are addictive which means that $\frac{dc_1}{dh}=\frac{d\ell}{dh}>0$ then they induce satiation on $c_2$ as $\frac{dc_2}{dh}=\frac{dB(1-\ell)}{dh}<0$.

\subsection{Dynamic Programming}
The value function of the problem is
\begin{equation}\label{valuefunction}
V(h_0)=\sup_{\ell(\cdot) \in \mathcal{A}(h_0)} \int_0^{+\infty}e^{-\rho t}U(\ell(t), B(1-\ell(t)),h(t))dt
\end{equation}
subject to the following statee equation
\begin{equation}\label{eq:state}
\dot h=\phi(\ell-h) \quad h(0)=h_0,
\end{equation}
where
$
\mathcal{A}(h_0)$ is defined in \eqref{eq: admissible_controls}.
The HJB equation associated to optimal control problem \eqref{optproblem} is
\begin{equation}\label{HJB}
\rho V=\max_{\ell \in [\underline \ell, \bar \ell]} H_{CV}(\ell,h,DV)
%\left\{\phi(\ell-h)DV+\frac{\left(\frac{\ell}{h^\gamma}\right)^{1-\sigma}}{1-\sigma}+\tilde B \frac{(1-\ell)^{1-\sigma}}{1-\sigma}\right\}
\end{equation}
where $H_{CV}(\ell,h,p)$ is defined in \eqref{eq:H_CV}.

 We are going to prove that the value function is a viscosity solution to \eqref{HJB}. First we recall the definition of viscosity solution.

\begin{definition}
A function $u \in C([0,+\infty))$ is a \textit{viscosity subsolution} of \eqref{HJB} if, for any $\phi \in C^1([0,+\infty))$
\begin{equation}\label{eq:subsol}
\rho V(h_0)\leq \max_{\ell \in [\underline \ell, \bar \ell]} H_{CV}(\ell,h_0,D\phi(h_0))
\end{equation}
at any local maximum point $h_0 \in [0,+\infty)$ of $u -\phi$. Similarly, $ u \in C^([0,+\infty))$ is a \textit{viscosity supersolution} of \eqref{HJB} if, for any $\phi \in C^1([0,+\infty))$
\begin{equation}\label{eq:supersol}
\rho V(h_0)\geq \max_{\ell \in [\underline \ell, \bar \ell]} H_{CV}(\ell,h_0,D\phi(h_0))
\end{equation}
at any local minimum point $h_0 \in [0,+\infty)$ of $V-\phi$. Finally $u$ is a \textit{viscosity solution} of \eqref{HJB} if it is simultaneously a viscosity sub- and supersolution.
\end{definition}

\begin{proposition}
Let $V$ be the value function defined in \eqref{valuefunction}. Then
\begin{itemize}
\item[(i)]  $V$ is negative, $ V(0^+)>-\infty$,  and $V$ is decreasing and continuous in $[0,+\infty)$.
\item[(ii)] $V$ is a viscosity solution in $(0,+\infty)$ of the HJB equation \eqref{HJB}.

\item[(iii)]  $V$ is locally Lipschitz in $(0,+\infty)$.
\item[(iv)] $V(+\infty)=-\infty.$
\end{itemize}
\end{proposition}

\begin{proof}
(i). Since $\sigma >1$, it is straightforward to see that $V<0$. Moreover
$$
V(0^+)\geq \int_0^{+\infty} e^{-\rho t}\left(\frac{\underline \ell^{1-\sigma}}{1-\sigma}e^{-\phi\gamma(\sigma-1)t}(e^{\phi t}-1)^{\gamma(\sigma-1)} +\tilde B\frac{(1-\bar \ell)^{1-\sigma}}{1-\sigma}\right)dt>-\infty
$$
from which we conclude the second inequality.

Note first that the set of control $\mathcal{A}(h_0)$ is independent of $h_0$, since
$$
h(t)=e^{-\phi t}\left[h_0+\phi\int_0^t\ell(s)e^{\phi s} ds\right]\geq 0
$$
since $h_0\geq 0, \phi>0$ and $\ell \in [\underline \ell, \bar \ell]$.

Then, it is easy to see that  $V(\cdot)$ is decreasing since the set of controls $\mathcal{A}(h_0)$ is independent of $h_0$ while the utility function decreases (since $\gamma>0$ and $\sigma >1$) in $h(t)$, hence in $h_0$ as the equation for $h(t)$ is linear. The continuity follows by sequences by using that $U(\ell(s), B(1-\ell(s)), \cdot )$ for $s>0$ is locally Lipschitz  in $(0,+\infty)$ since $\gamma(\sigma-1)\geq1$ and  the continuity of $h(t)$ for all $t>0$ with respect to the initial datum.

(ii). Note first that the state constraint can be easily proved to hold for every control $\ell$ since $\phi>0$ and $\underline \ell>0$. Then the fact that $V$ is a viscosity solution of the HJB equation \eqref{HJB} follows by standard arguments in viscosity theory (see for example \cite{BCD} Chapther III Proposition $2.8$). However, differently from the standard case (see \cite{BCD}) here $U(\ell(s), B(1-\ell(s)),\cdot)$ for $s>0$ is not uniformly continuous but just locally Lipschitz in $(0,+\infty)$ if $\gamma(\sigma-1)\geq1$. We give the proof for completeness.  For convenience of notation, in the following proof we will use the notation $h^{h_0,\ell}$ to denote the trajectory solution to \eqref{eq:state}. First note that the Dynamic Programming Principle (DPP) can be proved as in Proposition $2.5$ Chapter $3$ of \cite{BCD}, that is
$$
V(h_0)=\sup_{\ell(\cdot) \in \mathcal{A}(h_0)} \left \{\int_0^{t}U(\ell(s), B(1-\ell(s)), h^{h_0,\ell}(s))e^{-\rho s}ds+V(h^{h_0,\ell}(t)) e^{-\rho t}\right\}.
$$
Now we prove that the value function is a viscosity solution of \eqref{HJB}.  First we prove that $V$ is a subsolution.
 Let $\phi \in C^1([0,+\infty))$ and $h_0$ be a local maximum point of $V-\phi$, that is, for some $r>0$
\begin{equation}\label{eq:max}
V(h_0)-V(z)\geq \phi(h_0)-\phi(z), \quad \mbox{ for all } z \in B(h_0,r)\cap [0,+\infty).
\end{equation}

Then for each $\epsilon>0$ and $t>0$ by the inequality $\leq$ in the DPP, there exists $\tilde \ell \in \mathcal{A}(h_0)$ (depending on $\epsilon$ and $t$) such that
\begin{equation}\label{VDPP}
V(h_0)\leq \int_0^t U(\tilde \ell(s), B(1-\tilde \ell(s)),h^{h_0,\ell}(s))e^{-\rho s}ds+V(h^{h_0,\ell}(t))e^{-\rho t}+t\epsilon.
\end{equation}
Since $\gamma (\sigma-1)>1$ we have $U(\ell(s), B(1-\ell(s)), \cdot)$ is locally Lipschitz in $(0,+\infty)$ for any $s\geq 0$. For s enough small  we can suppose that $h^{h_0,\ell}(s) \in B(h_0,r)$.Then
\begin{eqnarray*}
|U(\tilde \ell(s), B(1-\tilde \ell(s)), h^{h_0,\ell}(s))-U(\tilde \ell(s), B(1-\tilde \ell(s)), h_0)|&\leq& C_r|h^{h_0,\ell}(s)-h_0|\\ &\leq & C_r r.
\end{eqnarray*}
Then the integral in the righthand side of \eqref{VDPP} can be rewritten as
$$
\int_0^tU(\tilde \ell(s), B(1-\tilde \ell(s)), h_0)e^{-\rho s}ds+o(t), \quad \mbox{ as } t \to 0
$$
where $o(t)$ indicates a function $g(t)$ such that $\lim_{t\to 0^+}\frac{|g(t)|}{t}=0$.
Then \eqref{eq:max} with $z=h^{h_0,\ell}(t)$ and \eqref{VDPP} give
\begin{equation}\label{eq:dpphjb}
\phi(h_0)-\phi(h^{h_0,\ell}(t))-\int_0^t U(\tilde \ell(s), B(1-\tilde \ell(s)), h_0)e^{-\rho s}ds+V(h^{h_0,\ell}(t))(1-e^{-\rho t})\leq t\epsilon+o(t)
\end{equation}
Moreover
\begin{eqnarray}\label{eq:estphi}
\phi(h_0)-\phi(h^{h_0,\tilde\ell}(t))\nonumber &=&-\int_0^t\frac{d}{ds}\phi(h^{h_0,\tilde\ell}(s)) ds\\ \nonumber&=& -\int_0^t D\phi (h^{h_0,\ell}(s))\cdot \phi (\tilde \ell(s)-h^{h_0,\tilde\ell}(s))ds\\ &=& -\int_0^t D\phi(h_0)\cdot \phi (\tilde \ell(s)-h^{h_0,\ell}(s))ds + o(t)
\end{eqnarray}
Plugging \eqref{eq:estphi} into \eqref{eq:dpphjb} and adding $\pm  \int_0^t U(\tilde \ell(s), B(1-\tilde \ell(s)), h_0)ds$ we get
\begin{multline*}
 \int_0^t \left\{-D\phi(h_0)\cdot \phi (\tilde \ell(s)-h^{h_0,\ell}(s))-U(\tilde \ell(s), B(1-\tilde \ell(s)), h_0)\right\}ds \\+\int_0^t U(\tilde \ell(s), B(1-\tilde \ell(s)), h_0)(1-e^{-\rho s})ds+V(h^{h_0,\ell}(t))(1-e^{-\rho t})\leq t\epsilon+o(t).
\end{multline*}
The term in brackets in the first integral is estimated from below by
$$
\min_{\ell \in [\underline \ell, \bar \ell]}\left\{-D\phi(h_0)\cdot \phi (\ell-h_0)-U(\ell, B(1-\ell), h_0)\right\}
$$
and the second integral is $o(t)$, so we can divide by $t$ and pass to the limit to get
$$
\min_{\ell \in [\underline \ell, \bar \ell]} \left \{-D\phi(h_0)\cdot \phi (\ell-h_0)-U(\ell, B(1-\ell), h_0)\right\}+\rho V(h_0)\leq \epsilon
$$
where we have used the continuity of $V$ at $h_0$ and of $h^{h_0,\ell}$ at $0$. Note that the previous inequality is equivalent to \eqref{eq:subsol}. Since $\epsilon$ is arbitrary, the proof that $V$ is a subsolution is complete.

Now we prove that $V$ is a supersolution to \eqref{HJB}. Let $\phi \in C^1([0, + \infty))$ and $h_0$ be a local minimum point of $V-\phi$, that is, for some $r>0$ ,
$V(h_0)-V(z)\leq \phi(h_0)-\phi(z)$ for all $z \in B(h_0,r)$. Fix an arbitrary $\ell \in [\underline \ell, \bar \ell]$ and let $h^{h_0,\ell}(t)$ be the solution corresponding to the constant control $\ell(t)=\ell$ for all $t$. For $t$ small enough $h^{h_0,\ell}(t) \in B(h_0,r)$ and then
$$
\phi(h_0)-\phi(h^{h_0,\ell}(t))\geq V(h_0)-V(h^{h_0,\ell}(t)) \quad \mbox{ for all } 0\leq t \leq t_0.
$$
By using the inequality "$\geq$" in the DPP, we get
$$
\phi(h_0)-\phi(h^{h_0,\ell}(t))\geq \int_0^{t}U(\ell, B(1-\ell), h^{h_0,\ell}(s))e^{-\rho s}ds+V(h^{h_0,\ell}(s)) (e^{-\rho t}-1).
$$
Therefore dividing by $t>0$ and letting $t\to 0$, we obtain, by the differentiability of $\phi$ and the continuity of $V, h^{h_0,\ell}$ and $U$
$$
-D\phi(h_0)\cdot (h^{h_0,\ell})'(0)=-D\phi(h_0) \cdot \phi (\ell-h_0)\geq U(\ell, B(1-\ell), h_0)-\rho V(h_0).
$$
Since $\ell \in [\underline \ell, \bar \ell]$ is arbitrary, we have proved that
$$
\rho V(h_0)+\min_{\ell \in [\underline \ell, \bar \ell]}\left\{-\phi(\ell-h_0)\cdot D\phi(h_0)-U(\ell, B(1-\ell), h_0)\right\}\geq 0,
$$
and being the previous inequality equivalent to \eqref{eq:supersol}, we get that $V$ is a viscosity supersolution to \eqref{HJB}.

(iii) Let $h_1>h_2$. For every $\delta>0$,  there exists $\ell_2(\cdot)\in \mathcal{A}$ such that
\begin{eqnarray*}
0\leq V(h_2)-V(h_1)&\leq& \int_0^{+\infty}e^{-\rho t}\left[u(h_2(t), \ell_2(t))-u(h_1(t), \ell_2(t))\right]dt+\delta\\ &=&\int_0^{+\infty}e^{-\rho t}\frac{\ell_2(t)^{1-\sigma}}{1-\sigma}\left[h_2(t)^{\gamma(\sigma-1)}-h_1(t)^{\gamma(\sigma-1)}\right]dt+\delta,
\end{eqnarray*}
where for convenience of notation  we denote by $h_1(\cdot), h_2(\cdot)$ the trajectories $h^{h_1, \ell_2}(\cdot), h^{h_2, \ell_2}(\cdot)$, respectively.
Let $$
a(t)=\int_0^te^{\phi s}\ell_2(s)ds.$$

Then by the Lagrange theorem there exists $\xi \in (h_2,h_1)$ such that
\begin{eqnarray*}
h_1(t)^{\gamma(\sigma-1)}-h_2(t)^{\gamma(\sigma-1)}&=& \left(e^{-\phi t}\left[h_1+a(t)\right]\right)^{\gamma(\sigma-1)}-\left(e^{-\phi t}\left[h_2+a(t) \right]\right)^{\gamma(\sigma-1)}\\ &=& e^{\phi \gamma(1-\sigma) t}(\xi+a(t))^{\gamma(\sigma-1)-1}(-\gamma(1-\sigma))(h_1-h_2).
\end{eqnarray*}
Then we have
\begin{eqnarray*}
h_2(t)^{\gamma(\sigma-1)}-h_1(t)^{\gamma(\sigma-1)}&\geq &e^{\phi \gamma(1-\sigma) t} (\xi+a(t))^{\gamma(\sigma-1)-1}\gamma(1-\sigma)(h_1-h_2)\\& \geq &e^{\phi \gamma(1-\sigma) t} (h_1+a(t))^{\gamma(\sigma-1)-1}\gamma(1-\sigma)(h_1-h_2)
\end{eqnarray*}
and then
$$
\frac{h_2(t)^{\gamma(\sigma-1)}-h_1(t)^{\gamma(\sigma-1)}}{1-\sigma}\leq e^{\phi \gamma(1-\sigma) t} (h_1+a(t))^{\gamma(\sigma-1)-1}\gamma(h_1-h_2)
$$
Note
$$
h_1+a(t)\leq h_1-\frac{\bar \ell}{\phi}+\frac{\bar \ell}{\phi}e^{\phi t}\leq2\max\left\{h_1-\frac{\bar \ell}{\phi}, \frac{\bar \ell}{\phi}\right\}e^{\phi t}:=ce^{\phi t}.
$$
Note that $c$ depends on $h_1$.
Then
\begin{eqnarray}\label{eq:disV}
0\leq V(h_2)-V(h_1)\leq \left(\underline \ell^{1-\sigma}\gamma c^{\gamma(\sigma-1)-1} \int_0^{+\infty}e^{-(\rho +\phi)t} dt\right)(h_1-h_2)+\delta
\end{eqnarray}

Since \eqref{eq:disV} holds for every $\delta>0$, we have
$$
0\leq \frac{V(h_2)-V(h_1)}{h_1-h_2}\leq \underline \ell^{1-\sigma}\gamma (\rho + \phi)^{-1}  c^{\gamma(\sigma-1)-1},
$$
completing the proof.

(iv) We observe that $u_1, u_2$ are increasing functions and $v$ is a decreasing function. So
\[u_1(h^{h_0,\underline{\ell}}_t)\leq u_1(h^{h_0,\ell}_t)\leq u_1(h^{h_0,\bar{\ell}}_t), \quad u_2(\underline{\ell})\leq u_2(\ell)\leq u_2(\bar{\ell})\]
\[v(\bar{\ell})\leq v(\ell)\leq v(\underline{\ell})\]
and then,
\[u_2(\ell)u_1(h^{h_0,\underline{\ell}}_t)\leq u_2(\bar{\ell})u_1(h^{h_0,\underline{\ell}}_t)\]
We can estimate the value function,
\[V(h_0)\leq \int_0^\infty e^{-\rho t} u_2(\bar{\ell})u_1(h^{h_0,\underline{\ell}}_t) dt+v(\underline{\ell})=\int_0^\infty e^{-\rho t} u_2(\bar{\ell})(e^{-\phi t}(h_0-\underline{\ell})+\underline{\ell}))^{-\gamma(1-\sigma)} dt+v(\underline{\ell}).\]
The right hand side converges to $+\infty$ and we conclude that $V(+\infty)=-\infty$.
\end{proof}

\section{Conclusion}
In this paper we have studied a problem for a social planner dealing with two good's consumptions, one of the two related to habits stock, and where the utility function is not jointly concave in the state and controls variable. First, the problem is studied when the utility on the consumption related to the good without habits is neglected. The complete problem is then studied by a dynamic programming approach: we have proved that the value function is a viscosity solution of the HJB equation and some qualitative properties on the value function are derived. These results constitute a a preliminary and necessary step for the more elaborate model where the lockdown (of random arrival and length) is investigated.

\newpage
%\nocite{*}
\bibliographystyle{alpha}
\bibliography{biblio}
\end{document}